\documentclass[12pt,a4paper]{article}

\usepackage{amsmath,amssymb,amsthm,epsfig}

\usepackage[english]{babel}

\usepackage[small,bf,hang]{caption} 

\usepackage{subfig}

\usepackage{algorithm, algorithmic}

\newtheorem{theorem}{Theorem}

\newtheorem{lemma}{Lemma}

\newtheorem{observation}{Observation}
\newtheorem{conjecture}{Conjecture}

\newenvironment{dedication}
{\begin{quotation}\begin{center}\begin{em}}
{\end{em}\end{center}\end{quotation}}

\usepackage[usenames, dvipsnames]{color}

\begin{document}

\title{Switching 3-Edge-Colorings of Cubic Graphs}

\author{
Jan Goedgebeur\\
\small Department of Computer Science\\[-0.8ex]
\small KU Leuven campus Kulak\\[-0.8ex]
\small 8500 Kortrijk, Belgium\\[-0.8ex]
\small and\\[-0.8ex]
\small Department of Applied Mathematics, Computer Science and Statistics\\[-0.8ex]
\small Ghent University\\[-0.8ex]
\small 9000 Ghent, Belgium\\[-0.8ex]
\small\tt jan.goedgebeur@ugent.be\\
\\
Patric R. J. \"Osterg\aa rd\\
\small Department of Communications and Networking\\[-0.8ex]
\small Aalto University School of Electrical Engineering\\[-0.8ex]
\small P.O.\ Box 15400, 00076 Aalto, Finland\\[-0.8ex]
\small\tt patric.ostergard@aalto.fi
}

\date{}

\maketitle

\begin{dedication}
In loving memory of Johan Robaey
\end{dedication}

\begin{abstract}
  The chromatic index of a cubic graph is either 3 or 4.
  Edge-Kempe switching, which can be used to transform
  edge-colorings, is here considered for 3-edge-colorings
  of cubic graphs. 
  Computational
  results for edge-Kempe switching of cubic graphs up to order 30 and 
  bipartite cubic graphs up to order 36 are tabulated. Families of 
  cubic graphs of orders $4n+2$ and $4n+4$ with $2^n$ edge-Kempe
  equivalence classes are presented; it is conjectured that
  there are no cubic graphs with more edge-Kempe equivalence classes.
  New families of nonplanar bipartite cubic graphs with exactly
  one edge-Kempe equivalence class are also obtained.
  Edge-Kempe switching is further connected to cycle switching of
  Steiner triple systems, for which an improvement of the established 
  classification algorithm is presented.

\end{abstract}


 
\section{Introduction}

We consider simple finite undirected graphs without loops.
For such a graph $G = (V,E)$, the number of vertices $|V|$ is the
\emph{order} of $G$ and the number
of edges $|E|$ is the \emph{size} of $G$. If all vertices of
$G$ have the same number of neighbors, then $G$ is said to be \emph{regular}
and the number of neighbors is called the \emph{degree}. A graph that is regular
with degree $k$ is called $k$-\emph{regular}. The connected components
of 1-regular and 2-regular graphs consist of edges and cycles, respectively.
The smallest degree of regular graphs that cannot be easily characterized is 3.
A 3-regular graph is called \emph{cubic}. A regular graph with odd degree
necessarily has an even order.

A subgraph $G' = (V,E')$ of $G = (V,E)$ is said to be \emph{spanning} if each vertex
in $V$ is the endpoint of some edge in $E'$. A $t$-regular spanning subgraph
is called a $t$-\emph{factor}. A connected 2-regular spanning subgraph is a
\emph{Hamiltonian cycle}.
A \emph{decomposition} of a graph $G = (V,E)$ is a
set of subgraphs of $G$ whose edge sets partition $E$. A decomposition of a
regular graph into $t$-factors is called a $t$-\emph{factorization}.
A 1-factor is also known as a \emph{perfect matching}.

A $k$-\emph{edge-coloring} is a partition of the edges of a
graph into $k$ (color) classes so that no adjacent edges
are in the same class. Notice that we do \emph{not} label the
color classes in this work. Each edge-coloring with $k$
unlabeled colors and no empty color classes corresponds to $k!$ 
edge-colorings with labeled colors. This must be taken into account, for
example, when comparing counts from different studies---for 
example, our work and \cite{BH3}.

The smallest possible number of colors
in an edge-coloring is the \emph{chromatic index} of the graph.
No two edges of a \mbox{1-factor} are adjacent, and a decomposition
of a $k$-regular graph into $k$ \mbox{1-factors} is equivalent to
a $k$-edge-coloring of such a graph. Indeed, the terms 3-edge-coloring
and 1-factorization of a cubic graph are used interchangeably in the
current study.

As the edges of a $k$-regular graph cannot be
properly colored with fewer than $k$ colors, the chromatic index
of a $k$-edge-colorable $k$-regular graph is $k$.

The complement of a 1-factor of a cubic graph is a 2-factor. In
general, such a 2-factor may have cycles of arbitrary lengths,
but two cases of special interest here are when all cycles have
even length and when there is exactly one (that is, Hamiltonian)
cycle. The first case is related to edge-coloring.

\begin{theorem}
\label{thm:corr}
A cubic graph has a\/ $3$-edge-coloring iff it has a\/ $2$-factor
with cycles of even length.
\end{theorem}

We will refer to a 2-factor where all cycles have even length as an \textit{even\/ $2$-factor}.
Denoting the maximum vertex degree of the graph by $\Delta$,
by Vizing's theorem~\cite{V} the chromatic index is either $\Delta$ or $\Delta+1$.
This partitions graphs into two classes, called \emph{Class 1} and
\emph{Class 2}, respectively.

The chromatic index of a cubic graph is connected to several fundamental
problems. Tait~\cite{T} discovered that the four-color problem is equivalent to
showing that simple bridgeless connected planar cubic graphs have chromatic index 3.
Consequently, a 3-edge-coloring of a cubic graph is occasionally called a
\emph{Tait coloring}. 
For nonplanar graphs, it is possible for simple bridgeless connected cubic graphs
to have chromatic index 4, and such graphs are called \emph{snarks}, typically
also requiring the girth to be at least 5.
Snarks are related
to various important problems in graph theory, including the cycle double cover conjecture
and the 5-flow conjecture.

Given a 3-edge-coloring of a cubic graph, the problem of finding
more \mbox{3-edge-colorings} is here considered. This is accomplished in the
context of local transformations known as edge-Kempe switches.
The impact of edge-Kempe switching on Class-1 cubic graphs up to order 30 
and Class-1 bipartite cubic graphs up to order 36 is
computationally evaluated. 
Edge-Kempe switching is further 
involved in cycle switching of Steiner triple systems. It is discussed
how the cubic graph underlying cycle switching can be utilized to improve
the established algorithm for classifying Steiner triple systems.

Notice that similar transformations can also be carried out for \emph{vertex} 
colorings \cite{BBFJ,Mohar}, so a term that is not specifying the setting, such 
as ``Kempe switching'', may cause confusion.
We refer to~\cite{CGGST} for a survey on graph edge-coloring.

This paper is organized as follows. In Section~\ref{sect:sts} cycle switching
of Steiner triple systems is described. 
Cycle switching of Steiner triple systems is closely related
to and form one motivation for studying edge-Kempe switching of
\mbox{3-edge-colorings}
of cubic graphs, which is the topic of Section~\ref{sect:color}. 
An algorithm to compute the edge-Kempe equivalence classes is presented and is then used to computationally investigate edge-Kempe
switching of cubic graphs up to order 30 (and for various subclasses of cubic graphs such as bipartite cubic graphs, planar cubic graphs, and 3-connected planar cubic graphs up to higher orders).
Certain conjectures related to the number of 3-edge-colorings in cubic graphs are also computationally verified.
Furthermore, two families of cubic graphs of orders $4n+2$ and $4n+4$ with $2^n$ edge-Kempe
  equivalence classes are presented. It is conjectured that
  there are no cubic graphs with more edge-Kempe equivalence classes.
  New families of nonplanar bipartite cubic graphs with exactly
  one edge-Kempe equivalence class are also obtained.

\section{Cycle Switching of Steiner Triple Systems}

\label{sect:sts}

A \emph{Steiner triple system} (STS) is a pair $(X,\mathcal{B})$, 
where $X$ is a finite set of \emph{points} and $\mathcal{B}$ is a 
set of 3-subsets of points, called \emph{blocks}, such that 
every 2-subset of points occurs in exactly one block. The size of 
the point set is the \emph{order} of the STS, and an STS of order $v$ 
is commonly denoted by STS$(v)$. STSs exist iff the order is
\[
v \equiv 1\mbox{\ or }3\!\!\!\pmod{6}.
\]
For more information about Steiner triple systems, see~\cite{C}.

Given an STS $(X,\mathcal{B})$, consider a block $\{a,b,c\}\in\mathcal{B}$.
We define $\mathcal{B}_a$ (resp.\ $\mathcal{B}_b$, $\mathcal{B}_c$)
to be the set of blocks that contain $a$ (resp.\ $b$, $c$),
except $\{a,b,c\}$. All blocks in $\mathcal{B}_a \cup \mathcal{B}_b\cup\mathcal{B}_c$
intersect $X \setminus \{a,b,c\}$ in exactly 2 points and may therefore be considered
as the set of edges $E$ in the graph $G = (X \setminus \{a,b,c\},E)$.

As each pair of points is in some block of an STS, $G$ is 3-regular.
Moreover, the edges coming from the blocks containing a given
point $i \in \{a,b,c\}$ form a 1-factor of $G$, which we denote by $F_i$.
Hence we have a 3-edge-coloring and the chromatic index of $G$ is 3.

The fact that the local mapping of an STS to a 3-edge-coloring of $G$ is
reversible is the core of switching of Steiner triple systems: any other
\mbox{3-edge-coloring} of $G$ also gives an STS. Finding 3-edge-colorings
is not easy in general: the problem of determining whether a cubic graph
has chromatic index 3 is NP-complete~\cite{H}. However,
for a given 3-edge-coloring---which the original STS gives us---one may
consider the following subset of easily computed transformations.

A basic property of 1-factors is that the union of two 1-factors, 
$F_i \cup F_j$, \mbox{$F_i \cap F_j = \emptyset$,} forms an even
\mbox{2-factor}; we denote the number of cycles of the \mbox{2-factor} by $m$.
As each of the $m$ cycles of the 2-factor has two \mbox{1-factors}
(perfect matchings), $F_i \cup F_j$ has $2^m$ ordered pairs of
1-factors and $2^{m-1}$ unordered pairs of \mbox{1-factors,} that is,
1-factorizations. So a total of $2^{m-1}$ 3-edge-colorings can in this way be obtained from
one 3-edge-coloring. 
Any such edge-coloring can be obtained via a sequence
of switches each of which takes place in just one of the cycles.

A maximal path or a cycle with edges colored with two colors is called
an \emph{edge-Kempe chain} and forms a central part of the work by 
Kempe~\cite{K} on the four-color theorem. The edge-Kempe chains of a
3-edge-colored cubic graph
can only be cycles. Switching the colors of a 2-edge-colored
cycle (or maximal path, in the general case) 
is called an \emph{edge-Kempe switch} in graph theory.
The transformation of blocks of a Steiner triple system induced 
by an edge-Kempe switch of $G$ is
called a \emph{cycle switch} in design theory, and if the length
of the cycle is 4---the shortest possible---then it is called a
\emph{Pasch switch}, due to the name of the configuration of blocks involved.
All of these switches are reversible.

Obviously, for Steiner triple systems,
the block $\{a,b,c\}$ can be chosen arbitrarily, and the total
number of 2-factors that can be considered for switching an 
STS$(v)$ is $v(v-1)/2$.
If all of these 2-factors consist of just one cycle, whereby nothing new
can be found, the Steiner triple system is said to be \emph{perfect}.
No perfect Steiner triple systems of order less than or equal to 21
exist~\cite{K1}, and only sporadic examples of larger orders
are known~\cite{GGM}.

For all admissible orders up to 19, all Steiner triple systems of a
given order are connected via a sequence of cycle switches~\cite{GGM2,KMO}.
This property is obviously not possible for orders where perfect
systems exist. One possible way of handling such cases is discussed
in \cite{DGG}. Another way is to transform
3-edge-colorings into arbitrary 3-edge-colorings rather than considering
only edge-Kempe switches. Such transformations have earlier been
proposed by Petrenjuk and Petrenjuk~\cite{PP} and are more recently
treated in~\cite{CG}.
For example, the perfect STSs of order 25 and 33
listed in~\cite{GGM} can in this way be transformed into STSs that
are not perfect.
See~\cite{O} for a survey on switching codes and designs.

Finally, let us have a brief look at the algorithm 
used for classifying the Steiner triple systems of order 19 in~\cite{KO1} 
(see also ~\cite[Sect.\ 6.1]{KO2})
and how the cubic graph discussed above actually
plays a central role in a possible improvement of that algorithm.
The original algorithm consists of three main phases: (i) classification
of the structures of type
\begin{equation}
\label{eq:seed}
\{a,b,c\} \cup \mathcal{B}_a \cup \mathcal{B}_b\cup\mathcal{B}_c,
\end{equation}
(ii) extending those structures to STS$(v)$s in all possible ways, and
(iii) carrying out isomorph rejection amongst the STS$(v)$s thereby
obtained. 

The number of structures obtained in step (i) is the number
of 3-edge-colorings of cubic graphs of order $v-3$, up to
isomorphism. A core observation is now that extension in step (ii) 
only depends on the edges of the cubic graph $G = (X \setminus \{a,b,c\},E)$
coming from \eqref{eq:seed}, the precise blocks of
\eqref{eq:seed} being irrelevant. Consequently, it suffices to consider 
each 3-chromatic cubic graph of order $v-3$ once in step (ii). 

For example, for STS(19)s, the 14648 structures extended in step (ii)
in \cite{KO1} can be reduced as there are 4207 cubic graphs of order 16
and 3986 of these are 3-chromatic. 
For STS(21)s, the smallest open case as for classifying Steiner triple
systems, the three corresponding numbers are 219104 (by~\cite{KO1}), 
42110, and 40440. Notice that these are numbers for
arbitrary graphs, that is, disconnected graphs are also included.
Step (ii) is the most time-consuming part of
the algorithm, so the ratio of the numbers of instances is close to the
ratio of the total run times.

\section{Edge-Kempe Switching of 3-Edge-Colorings}
\label{sect:color}
The concept of edge-Kempe switching of 3-edge-colorings of cubic graphs was introduced
in the previous section; in this section we shall investigate various aspects of
such switching. 

Due to the connection with the four-color problem, planar graphs have
been of special interest in the work on edge-Kempe switching of
3-edge-colorings of cubic graphs, but also other particular classes of
graphs have been considered. Some of these results can found in
~\cite{BH1,BH2}. One extremal case is when there is exactly one
3-edge-coloring of a cubic graph. We will call such a graph 
\textit{uniquely\/ $3$-edge-colorable}. For a better understanding,
we include a proof of the following well-known result.

\begin{theorem}
\label{thm:unique}
If a cubic graph has a unique\/ $1$-factorization---that is, a unique\/ 
$3$-edge-coloring---then it has exactly three Hamiltonian cycles.
\end{theorem}

\begin{proof}
If the 3-edge-coloring is unique, edge-Kempe switching gives
nothing new, and the union of two 1-factors (colors) must form
a Hamiltonian cycle. This gives three Hamiltonian cycles.
The theorem now follows as further Hamiltonian cycles would
give different 3-edge-colorings.
\end{proof}

Thomason proved~\cite{T3} that the converse of Theorem~\ref{thm:unique} does not hold. In particular, he showed that all generalized Petersen graphs of the form GP$(6k+3, 2)$ with $k \geq 0$ contain exactly three Hamiltonian cycles, but that they are not uniquely 3-edge-colorable if $k \geq 2$. So Thomason's smallest counterexample is GP$(15, 2)$, which has order 30. 

In~\cite[Table~7]{GMZ} all cubic graphs up to order 32 with exactly three
Hamiltonian cycles are determined, 
and we generated all 3-edge-colorings of those graphs. 
In our work, we used two independent algorithms to generate all 3-edge-colorings
of graphs. The first algorithm---which was already used and tested before 
in~\cite{BGHM}---colors the edges one at a time, and 
the second algorithm generates all perfect matchings and then determines which 
combinations of three of these perfect matchings yield a partitioning of the 
edge set of the graph. 

The computations for the cubic graphs up to order 32 with exactly three Hamiltonian 
cycles led to the following observation.

\begin{observation} \label{obs:3hc}
There are exactly three cubic graphs up to\/ $32$ vertices with three Hamiltonian cycles that are not uniquely\/ $3$-edge-colorable: the generalized Petersen graph {\rm GP}$(15,2)$ of order\/ $30$ and two graphs of order\/ $32$.
\end{observation}

The two graphs of order 32 from Observation~\ref{obs:3hc} can be obtained from GP$(15, 2)$ by blowing up a vertex to a triangle.

If an edge-coloring can be obtained from another through a
sequence of edge-Kempe switches, then the edge-colorings are
said to be \emph{edge-Kempe equivalent} and belong to the same
\emph{edge-Kempe equivalence class}. Graphs with a unique
edge-coloring (cf.\ Theorem~\ref{thm:unique}) obviously have
just one edge-Kempe equivalence class.
Examples of other classes of graphs with exactly one edge-Kempe equivalence class include 
the prism graphs (also called circular ladder graphs), 
twisted (or crossed) prism graphs, and 
M\"obius ladders of order divisible by 4~\cite{BH1} (see Figure~\ref{fig:prisms}).

One may further consider automorphisms of the cubic graph to get another concept of equivalence
which merges some of the edge-Kempe equivalence classes. However, such equivalence is not considered
in this study (graph isomorphism was just briefly mentioned in Section~\ref{sect:sts}).

In~\cite{Mohar} Mohar poses the problem of classifying the bipartite cubic graphs 
that have exactly one edge-Kempe equivalence class.
Belcastro and Haas~\cite{BH1} provide a partial answer to this problem.

\begin{theorem}{\rm \cite[Corollary 4.3]{BH1}}
\label{thm:planar}
Every planar bipartite cubic graph has exactly one edge-Kempe equivalence
class.
\end{theorem}

In \cite{BH1} it is shown that twisted prism graphs 
(see Figure~\ref{fig:prisms}) have exactly one
edge-Kempe equivalence class, and such a graph of order 12
is the smallest nonplanar bipartite cubic 
graph with exactly one edge-Kempe equivalence class.
The next examples of nonplanar bipartite cubic graphs with
exactly one edge-Kempe equivalence class can be found for order 16,
one of which is again a twisted prism graph; the other two graphs
out of three are shown in Figure~\ref{fig:examples}. 

\begin{figure}[!htb]
\begin{center}
  \subfloat[]{\label{fig:hamster}\includegraphics[width=0.30\textwidth]{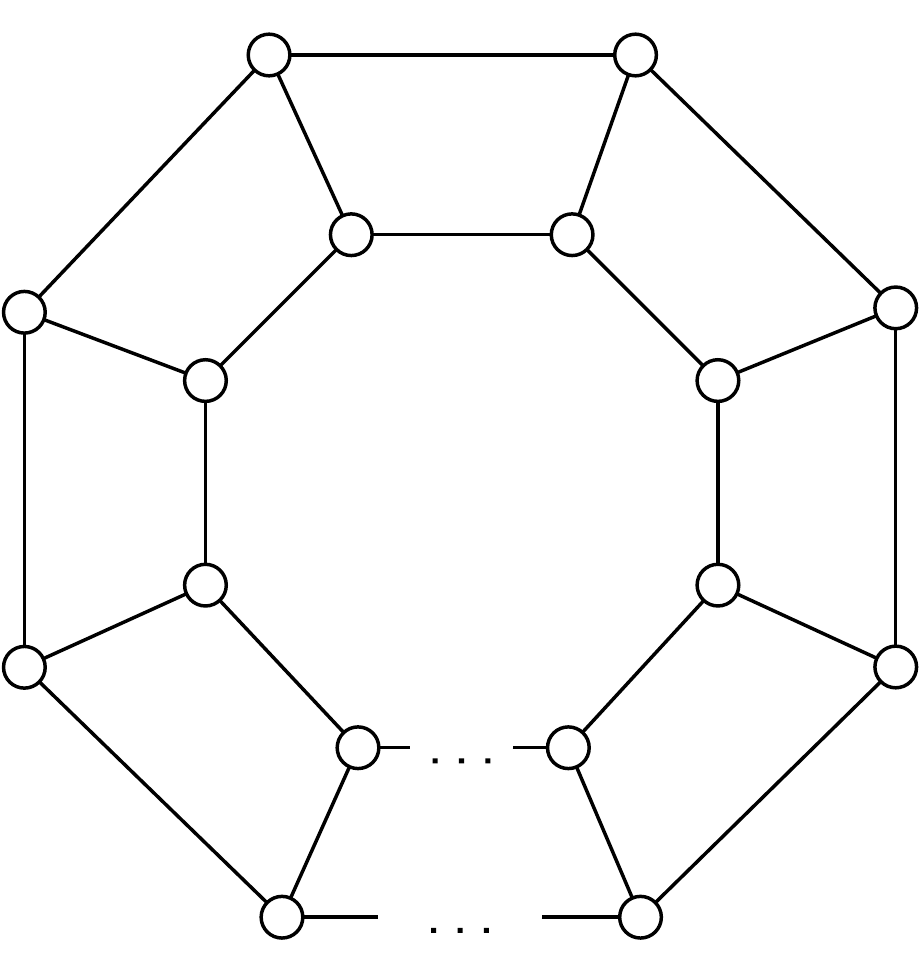}}   \quad
  \subfloat[]{\label{fig:twisted}\includegraphics[width=0.30\textwidth]{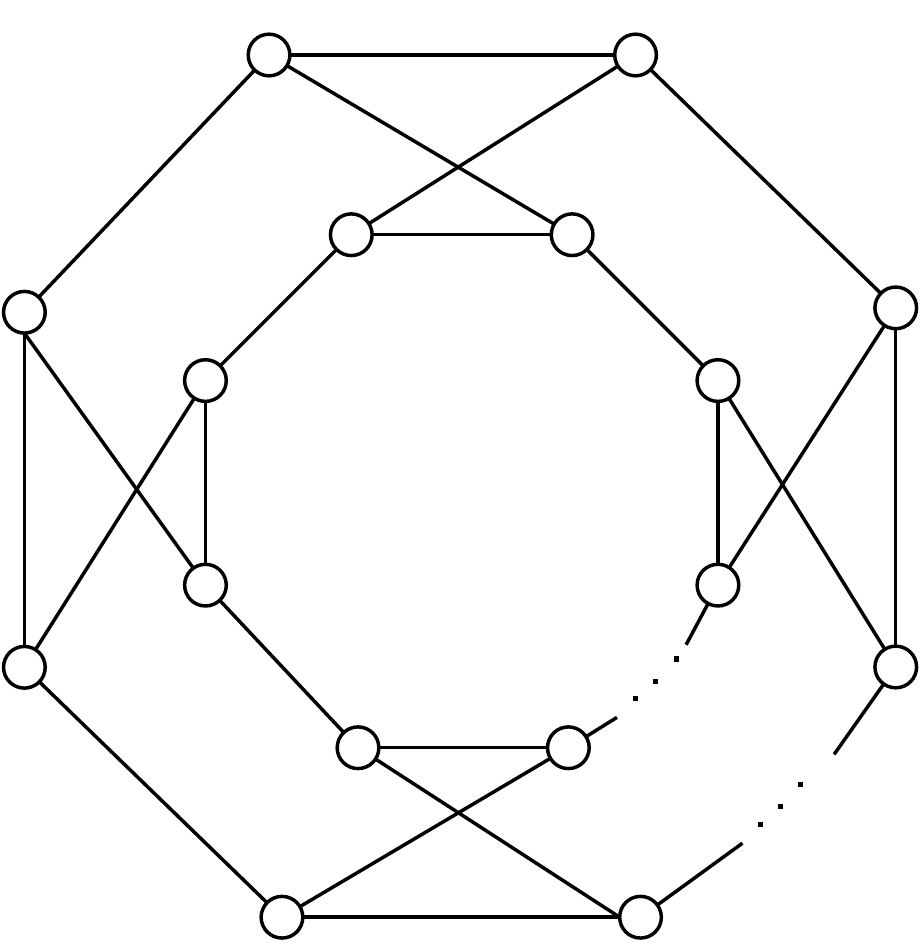}}   \quad
    \subfloat[]{\label{fig:mobius}\includegraphics[width=0.30\textwidth]{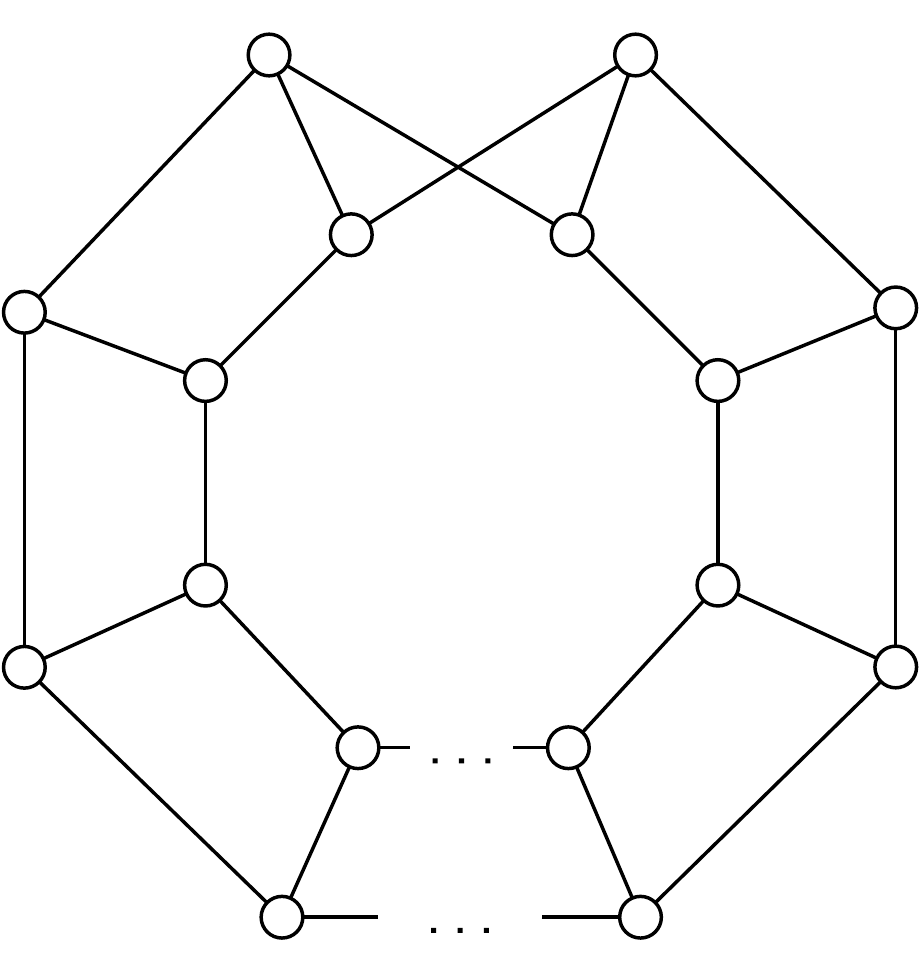}}   
	\caption{Prism graphs (a), twisted prism graphs (b), and 
M\"obius ladders (c)}
	\label{fig:prisms}
\end{center}
\end{figure}

\begin{figure}[!htb]
\begin{center}
\begin{minipage}[c]{.35\textwidth}
\centering   
   \subfloat[]{\label{fig:graph0}\includegraphics[width=1.0\textwidth]{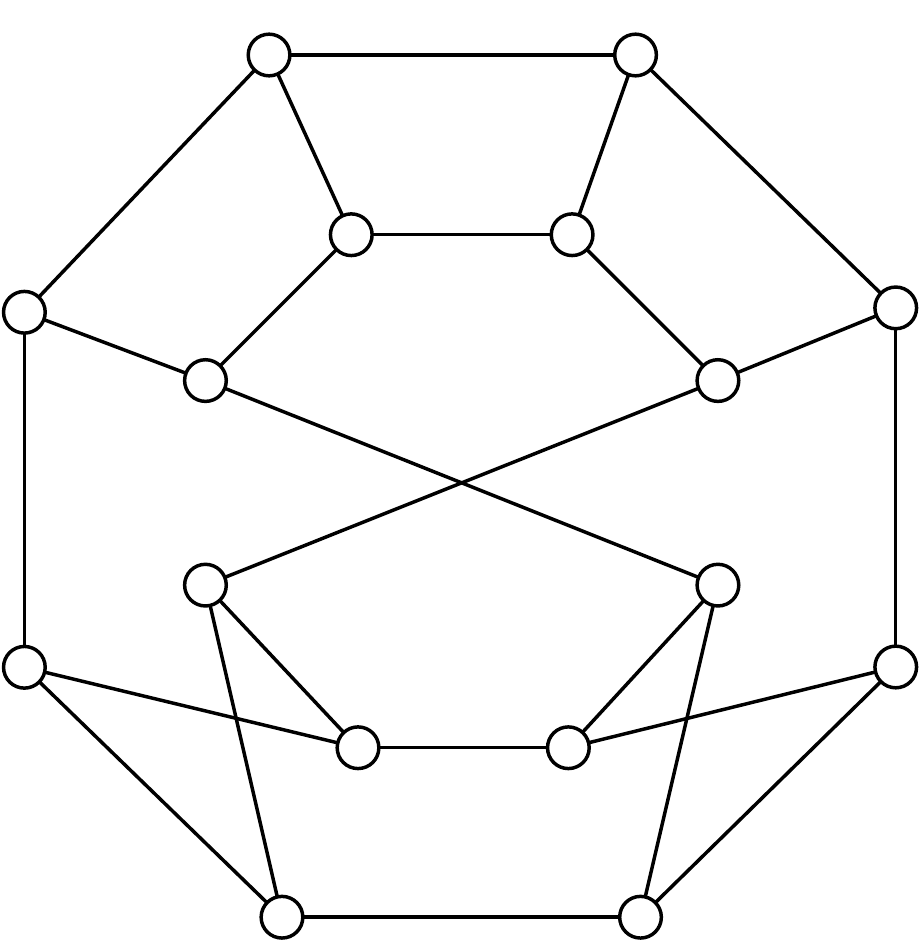}} 
\end{minipage}      
    \qquad  \qquad
\begin{minipage}[c]{.4\textwidth}
\centering   
\vspace{-0.8em} 
   \subfloat[]{\label{fig:graph2}\includegraphics[width=1.0\textwidth]{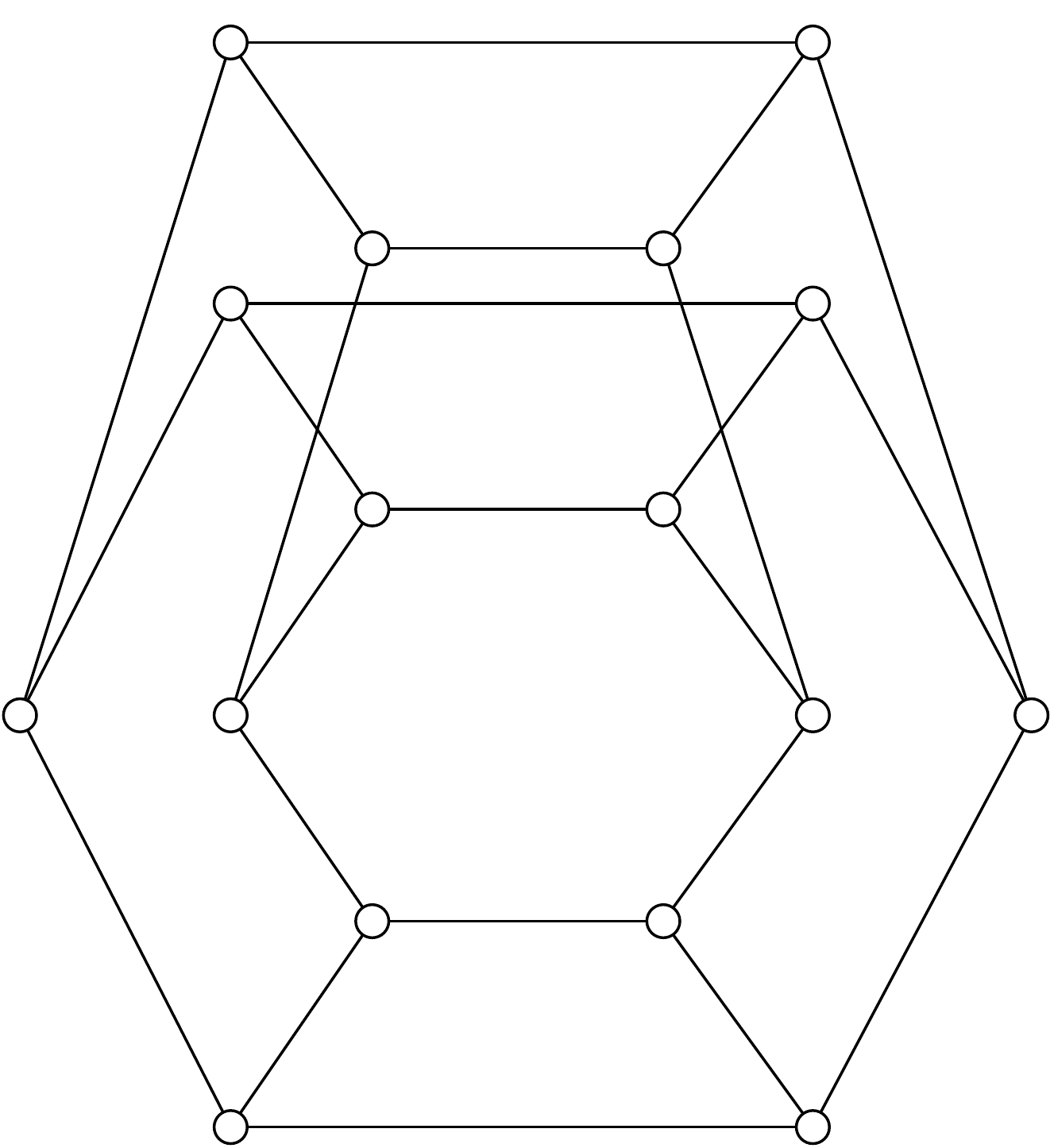}}   
\end{minipage}      
	\caption{Two nonplanar bipartite cubic graphs of order 16 with exactly one edge-Kempe equivalence class}
	\label{fig:examples}
\end{center}
\end{figure}

We shall now prove that the graph in Figure~\ref{fig:graph2}
belongs to an infinite family of graphs with exactly
one edge-Kempe equivalence class. Before defining the
family, let us consider some definitions and basic results
for ladder graphs. A \emph{ladder graph} $L_n$, $n \geq 1$, is a graph isomorphic to
$P_2 \times P_n$, and
a \emph{ladder graph with pendants} $LP_n$, $n \geq 1$,
is a graph obtained by adding four pendant edges and
vertices to $L_n$ as shown in Figure~\ref{fig:LP}. We say that the
pendant edges $e_a,e_b$ and $e_c,e_d$ form \emph{end pairs}
and $e_a,e_c$ and $e_b,e_d$ form \emph{side pairs}. 
(Arguably, $LP_n$ looks more like a real-world ladder than the
ladder graph $L_n$.)
An attempt to 3-edge-color a ladder graph with pendants immediately gives
the following lemma.

\begin{figure}[htb!]
	\centering
	\includegraphics[width=0.5\textwidth]{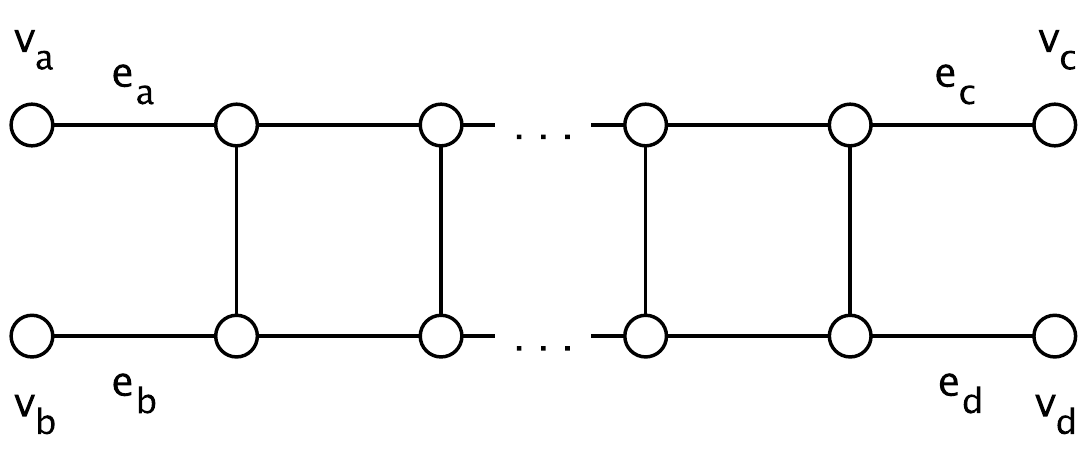}
	\caption{Ladder graphs with pendants}
	\label{fig:LP}
\end{figure}

\begin{lemma}
\label{lem:types}
Consider a graph $LP_n$ with $n$ even. If the edges of an end pair have
different fixed colors, then this uniquely defines the\/ $3$-edge-coloring of the entire graph, 
and the colors of the two edges of a side pair coincide. If the edges of an
end pair have the same fixed color, then also the colors of the edges of the other 
end pair coincide.
\end{lemma}

Depending on the edge-coloring of an end pair, we say that a 3-edge-coloring 
of a ladder graph with pendants has
type D (different) or S (same). For type S, it turns out that edge-Kempe switching is 
powerful in the following sense.

\begin{lemma}
\label{lem:ladder}
For $n\geq 2$ and a graph $LP_n$ of type S with
fixed colors for the end pairs, all\/ $3$-edge-colorings are in the same
edge-Kempe equivalence class.
\end{lemma}

\begin{proof}
The proof is constructive, that is, we shall show how an arbitrary \mbox{3-edge-coloring}
can be obtained from a given 3-edge-coloring via edge-Kempe switching. Starting from 
the edges of one end pair and their fixed colors, we shall proceed by considering the color of the
edge incident to both of those, a \emph{rung} of the graph $LP_n$. 
After this, the color of another pair of edges will be uniquely determined and 
this procedure will be iterated.

There are three possible situations when the current color of a rung is not the one we want,
depicted in Figure~\ref{fig:ladder_proof} (coloring proceeds from the bottom to the top). In the relevant
case, we
(a) switch the colors ABAB of the edges of a 4-cycle,
(b) switch the colors ABABAB of the edges of a 6-cycle,
(c) first switch the colors ACAC of the edges of the upper \mbox{4-cycle} and
then switch the colors ABAB of the edges of the lower 4-cycle.
Note that when we approach the end of the ladder, the colors of the 
pendant edges in that end restrict possible edge-colorings, so if the current 
color of the second to last rung is not the desired one, then only the 
situation in Figure~\ref{fig:ladder_a} is possible and edge-Kempe switching 
in the \mbox{4-cycle} finalizes the \mbox{3-edge-coloring.}
\end{proof}

\begin{figure}[!htb]
\begin{center}
  \subfloat[]{\label{fig:ladder_a}\includegraphics[width=0.16\textwidth]{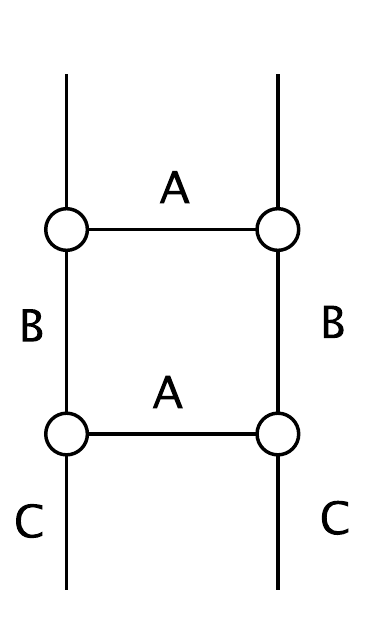}}   \qquad \qquad
  \subfloat[]{\label{fig:ladder_b}\includegraphics[width=0.16\textwidth]{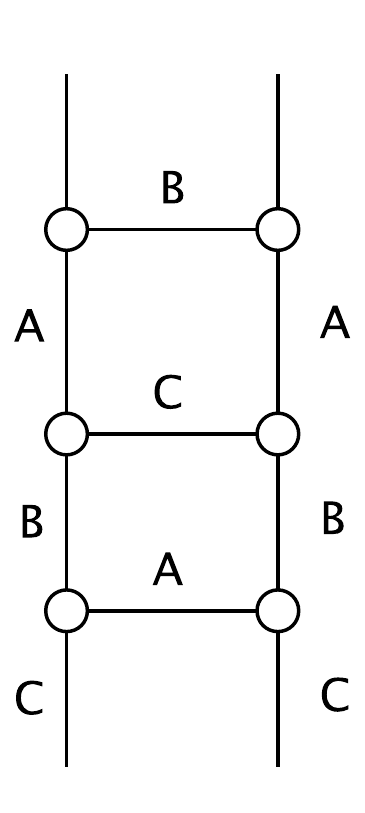}}   \qquad \qquad
    \subfloat[]{\label{fig:ladder_c}\includegraphics[width=0.16\textwidth]{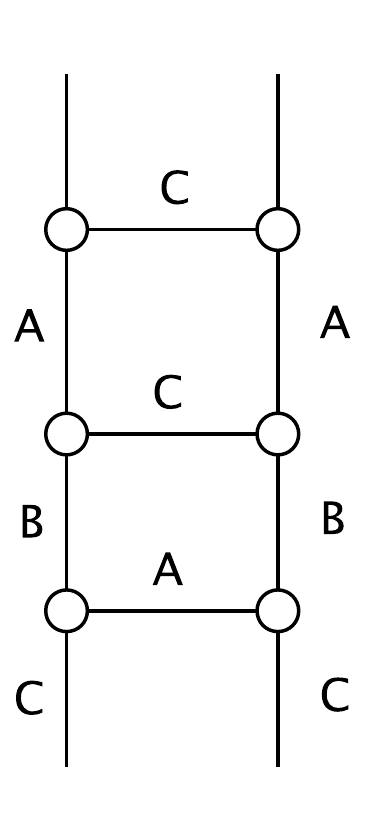}}   
	\caption{Ladder graphs with pendants used in the proof of Lemma~\ref{lem:ladder}}
	\label{fig:ladder_proof}
\end{center}
\end{figure}

Now we define an
\emph{eggbeater graph} $E_{i,j,k}$ as the graph obtained by taking 
three graphs that are isomorphic to $LP_i$, $LP_j$, and 
$LP_k$ and in the four cases of $v_a$, $v_b$, $v_c$, and $v_d$,
merging the corresponding vertices of the three graphs.
If $i$, $j$, and $k$ are even, then $E_{i,j,k}$ is bipartite and
nonplanar. Nonplanarity is shown by taking the endpoints of 
a rung from each of the three ladders
involved to get the vertices of a subgraph homeomorphic to $K_{3,3}$.
The graph in Figure~\ref{fig:graph2} is $E_{2,2,2}$.

\begin{theorem}
The graph $E_{i,j,k}$, where $i$, $j$, and $k$ are even, has exactly 
one edge-Kempe equivalence class of\/ $3$-edge-colorings.
\end{theorem}

\begin{proof}
Consider $E_{i,j,k}$ with $i$, $j$, and $k$ even. First consider all 3-edge-colorings
where one of the three ladder subgraphs with pendants, say $LP_i$, is of type D
(note that in later intermediate steps this subgraph can occasionally be of type S).
By the property of Lemma~\ref{lem:types}, this situation may be 
considered with a graph $G$ obtained by replacing $LP_i$ by two edges, 
$\{v_a,v_c\}$ and $\{v_b,v_d\}$, as long as we can force 
switching in the original graph to take place so that $e_a,e_c$ and 
$e_b,e_d$ are always in the same cycle when switching. Indeed, for type D this is
always the case, and for type S, by Lemma~\ref{lem:ladder} we can always do
edge-Kempe switching to get such a situation within the ladder subgraph.

Since $G$ is planar and bipartite, by Theorem~\ref{thm:planar} all 
3-edge-colorings of $G$ are in the same 
edge-Kempe equivalence class, and the same holds for the edge-colorings of $E_{i,j,k}$
of the given type.

Next we consider the 3-edge-colorings of $E_{i,j,k}$ where all three ladder subgraphs
with pendants are of type S. We shall constructively show how one such 3-edge-coloring
can be obtained from another. We first consider the edges incident to $v_a$ and $v_b$
(and then similarly for $v_c$ and $v_d$). We shall see how any two colors, say B and C,
can be transposed, that is, we consider edge-Kempe chains with B and C. Moreover, to be
able to treat the cases $v_a,v_b$ and $v_c,v_d$ separately, we need edge-Kempe chains that do 
not contain vertices from both of these pairs. Hence a chain going into a ladder
subgraph from $v_a$ should exit at $v_b$. Through proper switching we can make sure 
that this happens, and the cases in Figure~\ref{fig:ladder_proof} are relevant
also here. As a part of an edge-Kempe chain, we have in case (a) CAC and, after
switching BABA, CBC, and in case (b)/(c), CAC and CBCBC. Having done the edges incident
to $v_a$, $v_b$, $v_c$, and $v_d$---note that any permutation of three colors can be
obtained by at most two transpositions of colors---Lemma~\ref{lem:ladder} takes care 
of the rest.

Finally, we need to get between those main situations. For this, we
take a 3-edge-coloring where all three ladder subgraphs with pendants are of type
S and carry out switching to get a coloring that has an edge-Kempe chain through $v_a$ and
$v_c$ (which is possible due to Lemma~\ref{lem:ladder}).
Edge-Kempe switching then transforms the two involved ladder subgraphs with
pendants from type S to type D.
\end{proof}

Similar technniques can be used to obtain proofs and generalizations 
for other graphs encountered in our computational study, such as the
graph in Figure~\ref{fig:graph0}. It would be interesting to get 
more general characterization results for nonplanar bipartite cubic graphs 
with exactly one edge-Kempe equivalence class.

In \cite{BH1} Belcastro and Haas also state questions about
possible numbers of edge-Kempe equivalence classes
for 3-edge-colorable bipartite cubic graphs and for
arbitrary 3-edge-colorable cubic graphs. We shall now present
the algorithm used in our computational study of these and other
related questions.

In Algorithm~\ref{algo:kempe_equiv} we give the pseudocode of our algorithm 
to determine the edge-Kempe equivalence classes of a given cubic graph. 
Notice that we essentially have a problem of finding the connected components
of a graph where the vertices correspond to the 3-edge-colorings and the
edges to the edge-Kempe switches. There are several ways of handling such
implicit graphs, especially on the level of details (cf.\ \cite{KMO}).

In the algorithm we first compute and store 
each 3-edge-coloring as a triple of bitvectors $(a,b,c)$, where bit $i$ is 
set if the edge with number $i$ is colored in the color associated with 
the bitvector. Recall that colors are not labelled, and this is taken care
of by keeping the triple of bitvectors sorted, $a < b < c$.
The list of stored 3-edge-colorings is sorted so that binary search can be
used to find entries in logarithmic time
(alternatively, a hash function can be used).

The set of 3-edge-colorings forms the set of vertices of the implicit graph,
and edge-Kempe switching---which is done on-the-fly---gives the edges.
The connected components of the graph can be determined using 
Algorithm~\ref{algo:kempe_equiv}, which is a standard DFS algorithm \cite{HT}. 

A 3-edge-coloring imposes three even 2-factors, which are formed by the 
three pairs of 1-factors (color classes). In the core of the recursion,
edge-Kempe switching is carried out for every 
cycle of every 2-factor
of every 3-edge-coloring, with two exceptions for performance reasons: 
(i) we do not switch the
cycle that contains some fixed vertex $v$, and (ii) after performing a
switch, we do not carry out the reverse of the switch as this would
bring us to a edge-coloring that has already been visited. Exception (i) 
comes from the fact that switching all cycles of a 2-factor
does not change a 1-factorization. Hence switching one cycle of a
2-factorization gives the same result as switching at once all other cycles
in the 2-factorization, and therefore we may keep the edge-coloring of one
of the cycles, the one containing the chosen vertex $v$, fixed.
The fact that switching Hamiltonian cycles gives nothing new is a special
case of this observation.

\begin{algorithm}[htb!]
\caption{Edge-Kempe equivalence classes of a given cubic graph $G$}
{\tt edge-Kempe\_switch}(graph G)\\ \label{algo:kempe_equiv}
\vspace*{-5mm}
  \begin{algorithmic}[1]
	\STATE Generate and store all 3-edge-colorings of $G$
        \STATE Let $v$ be an arbitrary vertex of $G$ 
	\STATE Initialize all 3-edge-colorings of $G$ as not visited
        \FOR{every 3-edge-coloring $X$ of $G$}
        \IF{the 3-edge-coloring $X$ has not been visited} 
           \STATE {\tt edge-Kempe\_switch\_recursion}($X$,null,$v$) 
	\STATE Output the edge-Kempe equivalence class of $X$	
        \ENDIF
        \ENDFOR
  \end{algorithmic}
{\tt edge-Kempe\_switch\_recursion}(3-edge-coloring $X$, cycle $C$, 
vertex $v$)\\
\vspace*{-5mm}
  \begin{algorithmic}[1]
     \STATE{Mark the 3-edge-coloring $X$ as visited}
     \FOR{every even 2-factor $F$ imposed by $X$} 
     \FOR{every cycle $C' \neq C$ of $F$ that does not contain $v$}\label{algo:specialcase}
     \STATE Edge-Kempe switch $C'$ to get an edge-coloring $X'$ from $X$
     \IF{the 3-edge-coloring $X'$ has not been visited}
     \STATE Add $X'$ to the edge-Kempe equivalence class being constructed
     \STATE {\tt edge-Kempe\_switch\_recursion}($X'$,$C'$,$v$)
     \ENDIF
     \ENDFOR
     \ENDFOR
  \end{algorithmic}
\end{algorithm}

We used the \textit{snarkhunter}~\cite{BGM} program to generate all
cubic graphs and the two algorithms mentioned earlier---both of them,
to validate the results---to generate all 3-edge-colorings of those graphs.

An implementation of Algorithm~\ref{algo:kempe_equiv} can now be used to
determine the number of edge-Kempe equivalence classes of all connected cubic graphs 
up to what is computationally feasible. The numerical results of our 
computations up to order 30 are shown in Table~\ref{tab:kempe}.
For each order $n$, the number of graphs $N$, the number of graphs
with zero edge-Kempe equivalence classes $N_0$ (that is, graphs that are 
not 3-edge-colorable), the number of graphs
with exactly one edge-Kempe equivalence class $N_1$, the number of graphs with the 
maximum number of edge-Kempe equivalence classes $N_{\mbox{max}}$, and a list of the
number of edge-Kempe equivalence classes which occur are given.

In Tables~\ref{tab:bikempe}, \ref{tab:planar_kempe}, and~\ref{tab:3conn_planar_kempe} the results of the same computations are shown for connected bipartite cubic graphs, connected planar cubic graphs, and 3-connected planar cubic graphs, respectively. (We used \textit{plantri}~\cite{BM} to generate the 3-connected planar cubic graphs.)

\begin{table}[htb!]
\caption{Edge-Kempe equivalence classes for connected cubic graphs}
\label{tab:kempe}
\begin{center}
\small
\setlength{\tabcolsep}{4pt} 
\begin{tabular}{rrrrrl}\hline
$n$ & $N$ & $N_0$ & $N_1$ & $N_{\mbox{max}}$ & \# equiv.\ classes\\\hline
4 & 1 & 0 & 1 & 1 & 1 \\
6 & 2 & 0 & 1 & 1 & 1,2 \\
8 & 5 & 0 & 4 & 1 & 1,2 \\
10 & 19 & 2 & 9 & 1 & 0--2,4 \\
12 & 85 & 5 & 44 & 4 & 0--4 \\
14 & 509 & 34 & 188 & 3 & 0--5,8 \\
16 & 4 060 & 212 & 1 258 & 15 & 0--6,8 \\
18 & 41 301 & 1 614 & 8 917 & 7 & 0--10,16 \\
20 & 510 489 & 14 059 & 75 630 & 81 & 0--12,16 \\
22 & 7 319 447 & 144 712 & 680 055 & 25 & 0--18,20,32 \\
24 & 117 940 535 & 1 726 497 & 6 496 848 & 469 & 0--22,24,32 \\
26 & 2 094 480 864 & 23 550 891 & 63 963 867 & 111 & 0--36,40,64 \\
28 & 40 497 138 011 & 361 098 825 & 644 968 468 & 3 132 & 0--38,40,42--44,48,64 \\
30 & 845 480 228 069 & 6 137 247 735 & 6 606 598 953 & 588 & 0--60,62,64--66,68--70,\\
 & & & & & 72,73,80,128 \\
\hline
\end{tabular}
\end{center}
\end{table}

\begin{table}[htb!]
\caption{Edge-Kempe equivalence classes for connected bipartite cubic graphs}
\label{tab:bikempe}
\begin{center}
\small
\setlength{\tabcolsep}{4pt} 
\begin{tabular}{rrrrl}\hline
$n$ & $N$ & $N_1$ & $N_{\mbox{max}}$ & \# equiv.\ classes\\\hline
6  & 1  & 0  & 1  & 2 \\
8  & 1  & 1  & 1  & 1 \\
10  & 2  & 0  & 1  & 2,4 \\
12  & 5  & 2  & 1  & 1,2,4 \\
14  & 13  & 1  & 3  & 1--5,8 \\
16  & 38  & 6  & 2  & 1--4,8 \\
18  & 149  & 4  & 7  & 1--10,16 \\
20  & 703  & 24  & 13  & 1--8,10,12,16 \\
22  & 4 132  & 28  & 25  & 1--18,20,32 \\
24  & 29 579  & 140  & 67  & 1--16,18,20,24,32 \\
26  & 245 627  & 244  & 111  & 1--36,40,64 \\
28  & 2 291 589  & 1 026  & 453  & 1--30,32,34,36,40,48,64 \\
30 & 23 466 857 & 2 588 & 588 & 1--58,60,62,64--66,68--70,72,73,80,128 \\
32 & 259 974 248 & 10 066 & 3 112 & 1--52,54--56,58,60,62,64,66,68,70,72,80,96,128 \\
34 & 3 087 698 618 & 30 848 & 3 469 & 1--106,108--110,112--114,116,120,124,128--130,\\
& & & & 132,136--138,140,144--146,160,256 \\
36 & 39 075 020 582 & 117 304 & 22 832 & 1--104,106,108,110--112,114,116,118,120,124,\\
& & & & 128,130,132,136,138,140,144,146,160,192,256 \\
\hline
\end{tabular}
\end{center}
\end{table}

\begin{table}[htb!]
\caption{Edge-Kempe equivalence classes for connected planar cubic graphs}
\label{tab:planar_kempe}
\begin{center}
\begin{tabular}{rrrrrl}\hline
$n$ & $N$ & $N_0$ & $N_1$ & $N_{\mbox{max}}$ & \# equiv.\ classes\\\hline
4 & 1 & 0 & 1 & 1 & 1 \\
6 & 1 & 0 & 1 & 1 & 1 \\
8 & 3 & 0 & 1 & 1 & 1 \\
10 & 9 & 1 & 8 & 8 & 0,1 \\
12 & 32 & 3 & 28 & 1 & 0--2 \\
14 & 133 & 19 & 111 & 3 & 0--2 \\
16 & 681 & 98 & 556 & 27 & 0--2 \\
18 & 3 893 & 583 & 3 108 & 1 & 0--3 \\
20 & 24 809 & 3 641 & 19 368 & 1 & 0--4,10 \\
22 & 169 206 & 24 584 & 128 811 & 1 & 0--4,10 \\
24 & 1 214 462 & 174 967 & 897 475 & 7 & 0--5,7,10 \\
26 & 9 034 509 & 1 302 969 & 6 457 338 & 42 & 0--7,10 \\
28 & 69 093 299 & 10 038 834 & 47 603 292 & 1 & 0--8,10,11 \\
30 & 539 991 437 & 79 459 168 & 357 637 537 & 2 & 0--14,20 \\
\hline
\end{tabular}
\end{center}
\end{table}

\begin{table}[htb!]
\caption{Edge-Kempe equivalence classes for 3-connected planar cubic graphs}
\label{tab:3conn_planar_kempe}
\begin{center}
\small
\begin{tabular}{rrrrl}\hline
$n$ & $N$ & $N_1$ & $N_{\mbox{max}}$ & \# equiv.\ classes\\\hline
4  & 1  & 1  & 1  & 1 \\
6  & 1  & 1  & 1  & 1 \\
8  & 2  & 1  & 1  & 1 \\
10  & 5  & 1  & 1  & 1 \\
12  & 14  & 13  & 1  & 1,2 \\
14  & 50  & 47  & 3  & 1,2 \\
16  & 233  & 210  & 23  & 1,2 \\
18  & 1 249  & 1 096  & 1  & 1--3 \\
20  & 7 595  & 6 373  & 1  & 1--4,10 \\
22  & 49 566  & 39 860  & 1  & 1--4,10 \\
24  & 339 722  & 260 293  & 6  & 1--5,7,10 \\
26  & 2 406 841  & 1 753 836  & 31  & 1--7,10 \\
28  & 17 490 241  & 12 087 721  & 1  & 1--8,10,11 \\
30  & 129 664 753  & 84 809 873  & 2  & 1--14,20 \\
32 & 977 526 957 & 603 748 613 & 55 & 1--15,17,20 \\
34 & 7 475 907 149 & 4 350 914 098 & 1 & 1--20,23 \\
36 & 57 896 349 553 & 31 685 445 136 & 1 & 1--24,29,30,45 \\
38 & 453 382 272 049 & 232 863 258 652 & 1 & 1--26,28--30,32,35,40,45,100 \\
\hline
\end{tabular}
\end{center}
\end{table}

The graphs from Tables~\ref{tab:kempe} with the maximum number of edge-Kempe equivalence classes can be downloaded from~\cite{GO} and from the database of interesting graphs from the \textit{House of Graphs}~\cite{BCGM} by searching for the keywords ``maximum * edge-Kempe''.

In~\cite{BH1} a question is raised about the possible numbers of 
edge-Kempe equivalence classes. Our computational results show that
all positive integers up to 106 are covered. Hence we make the following
conjecture, which could also be made for more specific classes of
graphs.

\begin{conjecture}
For any positive integer $m$, there is a bipartite cubic graph
with exactly $m$ edge-Kempe equivalence classes.
\end{conjecture}

The following two theorems show the existence of an infinite family of graphs 
attaining the maximum number of equivalence classes in Tables~\ref{tab:kempe} 
and~\ref{tab:bikempe}. 

\begin{theorem}\label{thm:2mod4}
There are bipartite cubic graphs of order $4n+2$, $n \geq 1$, with 
$2^n$ edge-Kempe equivalence classes.
\end{theorem}

\begin{proof}
We shall show that there are graphs of order $4n+2$ with
exactly $2^n$ \mbox{3-edge-colorings} and for which all three 2-factors imposed by a
3-edge-coloring consist of just one cycle (which is a Hamiltonian cycle
in the original graph). This implies that the size of each edge-Kempe 
equivalence class is 1, and the theorem follows.

Let $G = (V,E)$ be the graph
shown in Figure~\ref{fig:2mod4}. Obviously, this graph is bipartite.
First consider the 1-factors that contain
$\{v_a,v_b\}$. The graph induced by $V \setminus \{v_a,v_b\}$ consists of 
$n$ subgraphs $K_{2,2}$ that are connected as shown in the figure.
By induction, we shall now show that this graph has $2^n$ 1-factors. 

If $n=1$, then the 1-factor either contains $\{v_{1,0},v_{2,0}\}$ and 
$\{v_{1,1},v_{2,1}\}$, or $\{v_{1,0},v_{2,1}\}$ and $\{v_{1,1},v_{2,0}\}$, that
is, there are $2 = 2^n$ possibilities. For arbitrary $n$, we have the same two
possibilities for the edges containing $v_{1,0}$ and $v_{1,1}$. The graph 
induced by $V \setminus \{v_a,v_b,v_{1,0},v_{1,1},v_{2,0},v_{2,1}\}$ consists of
$n-1$ subgraphs $K_{2,2}$ with the given structure. Hence the total number
of 1-factors is $2 \cdot 2^{n-1} = 2^n$.

Consider a 1-factor $F$ that contains $\{v_a,v_b\}$. To show that
the 2-factor $E \setminus F$ consists of exactly one cycle, we show
that it has a path from $v_a$ to $v_b$. Indeed, a vertex $v_{i,j}$ with
$i$ odd is an endpoint of $\{v_{i,j},v_{i+1,k}\} \in F$ and the neighbors in 
$E \setminus F$ are $v_{i-1,j}$ ($v_a$ if $i=1$) and $v_{i+1,1-k}$
(and similarly for the case of $i$ even).

Finally, a 2-factor containing $\{v_a,v_b\}$ must form a path $P$ from $v_a$ to
$v_b$ after removing $\{v_a,v_b\}$. Hence, for $1 \leq i \leq n$, $P$ must 
contain at least one vertex of type $v_{i,j}$. Let $i'$ be the smallest value of
$i$ for which $P$ contains exactly one vertex of type $v_{i,j}$. But a vertex
$v_{i',j'}$ cannot be in a 2-factor different from $P$ since at most one of its
neighbors in $G$ is not in $P$. This completes the proof.
\end{proof}

\begin{figure}[htb!]
	\centering
	\includegraphics[width=0.9\textwidth]{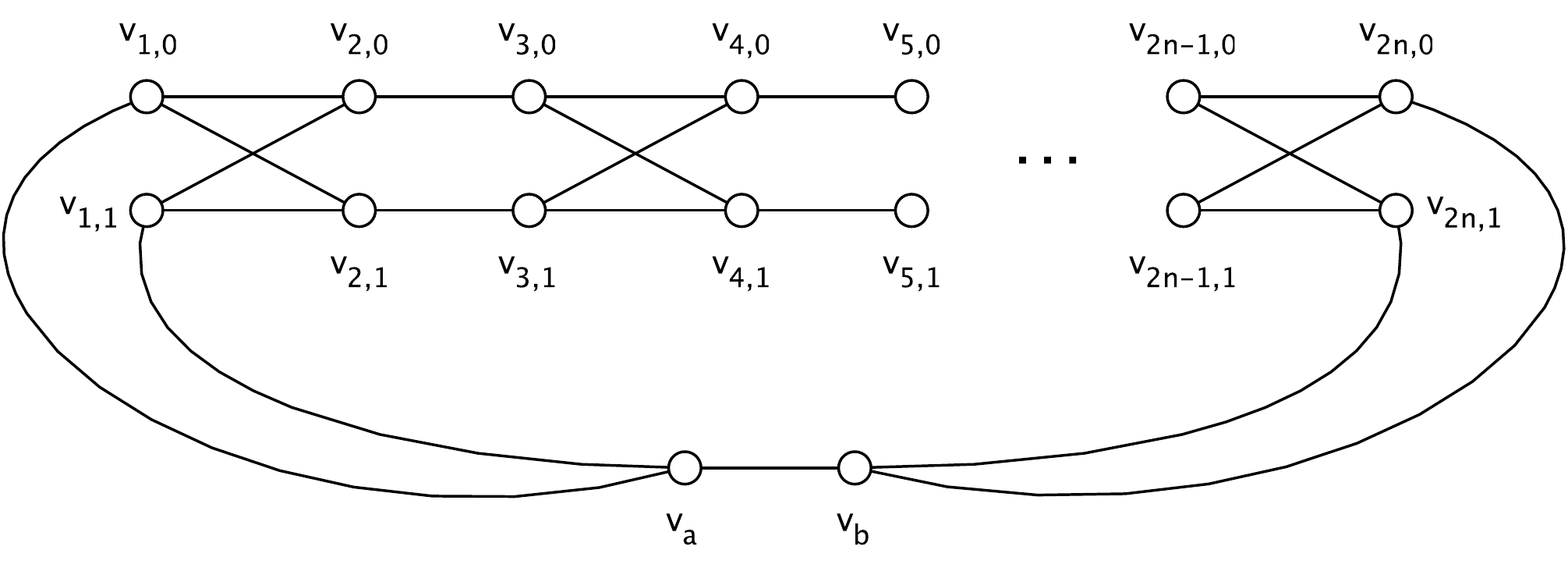}
	\caption{Graph for the proof of Theorem~\ref{thm:2mod4}}
	\label{fig:2mod4}
\end{figure}

The graphs for the proof of Theorem~\ref{thm:2mod4} are closely
related to those for the proof of Theorem~\ref{thm:0mod4}, so
large parts of the proofs are analogous.

\begin{theorem}\label{thm:0mod4}
There are bipartite cubic graphs of order $4n+4$, $n \geq 1$, with 
$2^n$ edge-Kempe equivalence classes.
\end{theorem}

\begin{proof}
We shall show that there are graphs of order $4n+4$ with
exactly $2^{n+1}$ \mbox{3-edge-colorings}, such that each
edge-Kempe equivalence class has size 2.

For $n \geq 2$, let $G = (V,E)$ be the bipartite graph
shown in Figure~\ref{fig:0mod4}, and 
consider the 1-factors that contain
$\{v_a,v_c\}$. Such a 1-factor contains either
$\{w_{1,0},w_{2,0}\}$ and 
$\{w_{1,1},w_{2,1}\}$, or $\{w_{1,0},w_{2,1}\}$ and $\{w_{1,1},w_{2,0}\}$
(2 possibilities), and it contains $\{v_b,v_d\}$. Finally, for the
edges incident to the vertices in the set
$V' = \{v_{1,0},v_{1,1},v_{2,0},v_{2,1},\ldots,v_{2n-2,0},v_{2n-2,1}\}$, there
are $2^{n-1}$ possibilities
by the argument in the proof of Theorem~\ref{thm:2mod4}.
This gives a total of $2 \cdot 2^{n-1} = 2^{n}$ possible 1-factors.

With arguments analogous to those in the
proof of Theorem~\ref{thm:2mod4}, for a 1-factor $F$ that contains 
$\{v_a,v_c\}$, the 2-factor $E \setminus F$ consists of exactly two 
cycles, one with vertex set $\{v_a,v_b\} \cup V'$ and the other
with vertex set $\{v_c,v_d,w_{1,0},w_{1,1},w_{2,0},w_{2,1}\}$. Hence 
there are 2 possibilities for the set of two 1-factors that impose
such a 2-factor, and the
total number of 3-edge-colorings is $2^n \cdot 2 = 2^{n+1}$.

As in the proof of Theorem~\ref{thm:2mod4} we may conclude that 
a 2-factor containing $\{v_a,v_c\}$ must, after removing 
$\{v_a,v_c\}$ and $\{v_b,v_d\}$, have a path from $v_a$ to
$v_b$ containing all vertices in $V'$ and another path from $v_c$ to $v_d$  
containing all vertices in $\{w_{1,0},w_{1,1},w_{2,0},w_{2,1}\}$, that
is, it must be Hamiltonian. Consequently, switching gives nothing
for the union of the 1-factor containing $\{v_a,v_c\}$ and any of the 
two other 1-factors. Switching with respect to the two other 1-factors
leads to an edge-Kempe equivalence class of size 2, so the total number
of edge-Kempe equivalence classes is $2^{n+1}/2 = 2^n$. For $n=1$,
the cubical graph of order 8 has 2 edge-Kempe equivalence classes,
which completes the proof.
\end{proof}

\begin{figure}[htb!]
	\centering
	\includegraphics[width=0.9\textwidth]{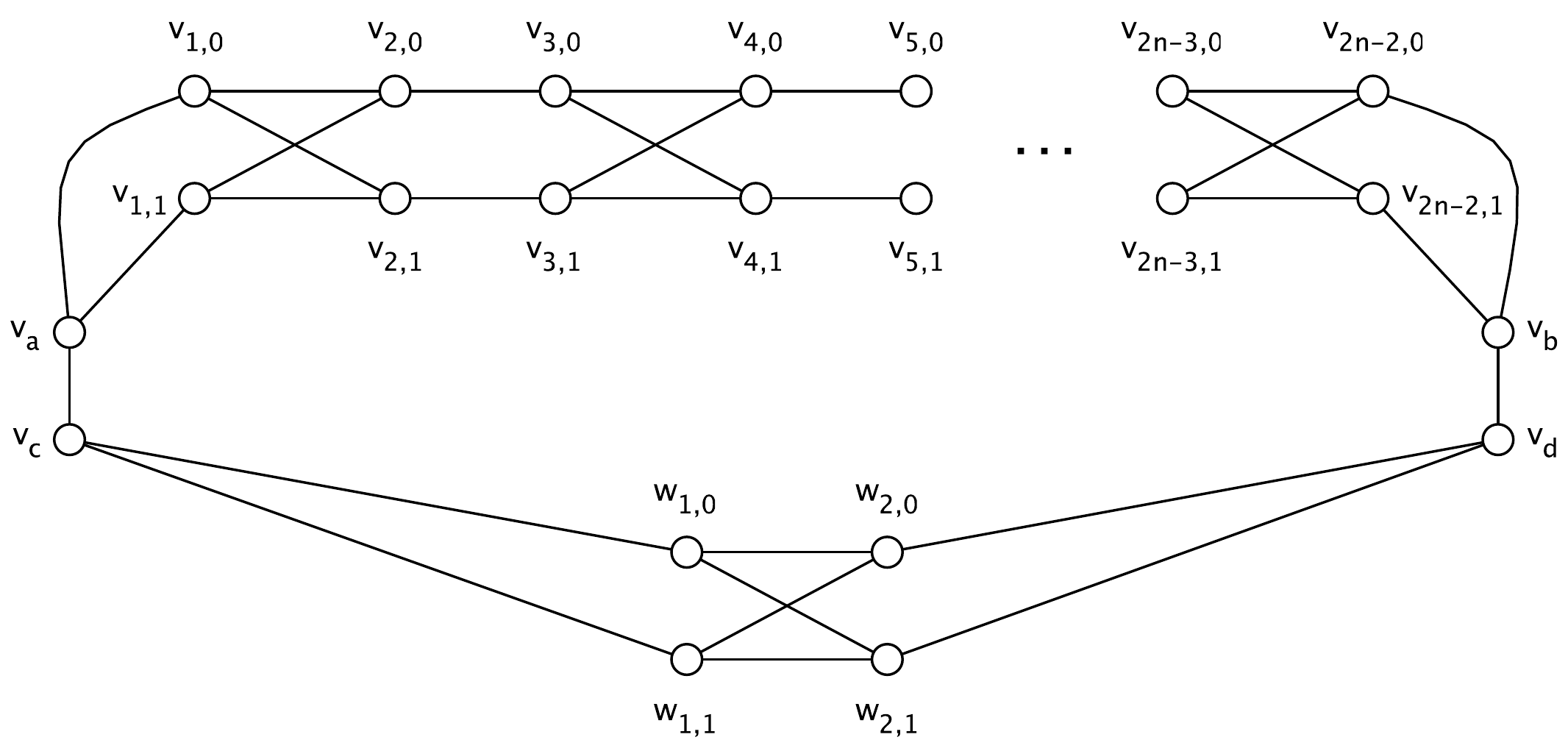}
	\caption{Graph for the proof of Theorem~\ref{thm:0mod4}}
	\label{fig:0mod4}
\end{figure}

We make the following conjecture.

\begin{conjecture}
There are no cubic graphs with more edge-Kempe equivalence classes than 
those in Theorems\/ $\ref{thm:2mod4}$ and\/ $\ref{thm:0mod4}$.
\end{conjecture}

Notice that there are also other graphs than those in
Theorems~\ref{thm:2mod4} and~\ref{thm:0mod4} with the same extremal
property.

The graphs in Figures~\ref{fig:2mod4} and~\ref{fig:0mod4} can be
obtained from twisted prism graphs (Figure~\ref{fig:prisms})
by a homomorphic mapping of, respectively, one and two $K_{2,2}$ subgraphs
to $K_{1,1}$ subgraphs. Indeed, the graphs have a 
similar circular structure of basic subgraphs as conjectured extremal graphs
for the maximum number of \mbox{3-edge-colorings}~\cite{BH3} 
(prism graphs and M\"obius ladders -- see Figure~\ref{fig:prisms})
and Hamiltonian cycles~\cite{E} (the family displayed in Figure~\ref{fig:eppstein_hc}). 

The conjectured maximum number of structures is  
$O(2^{n/2})$, $O(2^{n/3})$, and 
$O(2^{n/4})$
for 3-edge-colorings, Hamiltonian cycles, and edge-Kempe equivalence
classes, respectively.

\begin{figure}[htb!]
	\centering
	\includegraphics[width=0.5\textwidth]{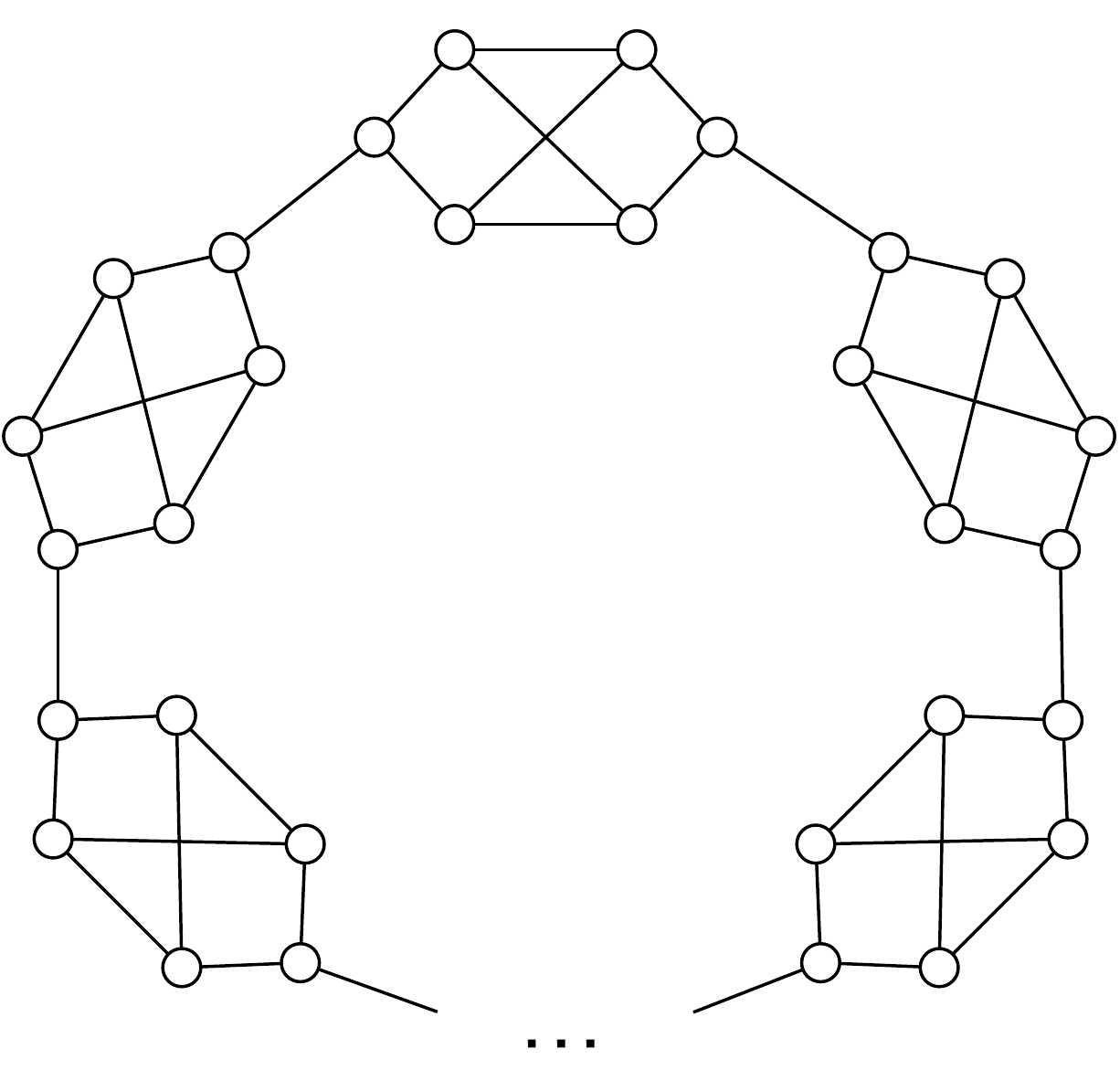}
	\caption{Graphs with many Hamiltonian cycles}
	\label{fig:eppstein_hc}
\end{figure}

In our study, we have obtained further results for graphs with many
\mbox{3-edge-colorings.} Bessy and Havet~\cite{BH3} show that a 
connected cubic graph of order $n$ can have at most $3\cdot 2^{n/2-3}$  3-edge-colorings,
and in \mbox{\cite[Conjecture 13]{BH3}} they state the following conjecture.

\begin{conjecture}
\label{conj:bessy_havet} 
The connected cubic graph of order $n$ with the largest number 
of\/ $3$-edge-colorings 
has $(2^{n/2-1} + 4)/3$ $3$-edge-colorings if $n/2$ is even and $(2^{n/2-1} + 2)/3$ 
$3$-edge-colorings if $n/2$ is odd. 
\end{conjecture}

We have been able to get further evidence for 
Conjecture~\ref{conj:bessy_havet}.

\begin{observation} \label{obs:bh}
Conjecture\/ $\ref{conj:bessy_havet}$ holds at least up to order\/ $30$.
\end{observation}


\subsection*{Acknowledgements}

We would like to thank Gunnar Brinkmann for useful suggestions. The observation
at the end of Section~\ref{sect:sts} dates back to discussions with Petteri Kaski after the
publication of \cite{KO2}.
Several of the computations for this work were carried out using the supercomputer infrastructure provided by the VSC (Flemish Supercomputer Center), funded by the Research Foundation Flanders (FWO) and the Flemish Government.


\end{document}